\documentclass{amsart}
\usepackage{amsmath}
\usepackage{amssymb}
\usepackage{amsthm}
\usepackage{enumerate}
\usepackage[pdftex]{graphicx}
\usepackage{caption}
\theoremstyle{definition}
\newtheorem{definition}{Definition}[section]
\theoremstyle{plain}
\newtheorem{lemma}[definition]{Lemma}
\newtheorem{theorem}[definition]{Theorem}

\newtheorem{corollary}[definition]{Corollary}
\theoremstyle{remark}

\makeatletter
\@namedef{subjclassname@2020}{
  \textup{2020} Mathematics Subject Classification}
\makeatother

\newcommand{\mycl}{\operatorname{cl}}

\begin{document}
\title[Lipchitz curve selection]{Lipchitz curve selection and its application to Thamrongthanyalak's open problem}
\author[M. Fujita]{Masato Fujita}
\address{Department of Liberal Arts,
	Japan Coast Guard Academy,
	5-1 Wakaba-cho, Kure, Hiroshima 737-8512, Japan}
\email{fujita.masato.p34@kyoto-u.jp}

\begin{abstract}
We solve an open problem posed in Thamrongthanyalak's paper on the definable Banach fixed point property.
A Lipschitz curve selection is a key of our solution.
In addition, we show a definable version of Caristi fixed point theorem. 
\end{abstract}



\maketitle	

\section{Introduction}\label{sec:intro}
Throughout, $\mathcal F=(F,<,+,\cdot,0,1,\ldots)$ is a definably complete expansion of an ordered field.
`Definable' means `definable with parameters'.
We recall basic notions.
$\mathcal F$ is \textit{definably complete} if every definable subset of $F$ has a supremum and infimum in $F \cup \{\pm \infty\}$ \cite{M}. 
$\mathcal F$ is \textit{locally o-minimal} if, for every $a \in F$ and every definable subset $X$ of $F$, there exists an open interval $I$ such that $a \in I$ and $X \cap I$ is a union of finitely many points and open intervals \cite{TV}.
We call a locally o-minimal structure $\mathcal F$ \textit{o-minimal} when we can choose $I=F$ \cite{D}.
An open core $\mathcal F^\circ$ of $\mathcal F$ is the structure on $F$ generated by open sets definable in $\mathcal F$ \cite{DMS,Fornasiero_locally,MS}. 

In Thamrongthanyalak's paper \cite{Thamrongthanyalak}, the Banach fixed point property (BFPP for short) is investigated.
A definable set $E$ has the \textit{BFPP} if every definable contraction on $E$ has a fixed point. 
Every nonempty definable closed set enjoys BFPP by \cite[1.4]{Thamrongthanyalak}.
$\mathcal F$ possesses the \textit{BFPP} (resp.~\textit{strong BFPP}) if every locally closed definable set (resp.~every definable set) having the BFPP is closed.
Thamrongthanyalak showed that structures having o-minimal open core enjoy the strong BFPP and, if $\mathcal F$ possesses the BFPP, $\mathcal F$ has a locally o-minimal open core. 
The following question is posed in \cite{Thamrongthanyalak}.

\begin{quote}
	If $\mathcal F$ is definably complete and possesses the strong BFPP, is it o-minimal?
\end{quote}

The following theorem answers the above question in the negative because non-o-minimal definably complete locally o-minimal expansion of an ordered field is already known \cite[Example 3.11]{Fornasiero_locally}.
\begin{theorem}\label{thm:local_case}
	A definably complete locally o-minimal expansion of an ordered field possesses the strong BFPP. 
\end{theorem}
We prove a Lipschitz curve selection lemma in Section \ref{sec:curve_selection}. 
Theorem \ref{thm:local_case} follows from the lemma in the same manner as \cite[Theorem A]{Thamrongthanyalak}.
A rough strategy of its proof is only given in the present paper.

Theorem \ref{thm:local_case} implies the `only if' part of the following corollary.
The `if' part was already proved as \cite[Theorem B]{Thamrongthanyalak}.
\begin{corollary}\label{cor:local}
 A definably complete expansion of an ordered field has a locally o-minimal open core if and only if it possesses the BFPP.
\end{corollary}

The Caristi fixed point theorem is a generalization of the Banach fixed point theorem \cite{Caristi}.
A metric space has the Caristi fixed point property if and only if it is complete \cite{Weston}.
We prove a similar equivalence holds in definably complete structures in Section \ref{sec:caristi}.
\begin{theorem}[Definable Caristi fixed point theorem]\label{thm:caristi}
	For a definable subset $X$ of $F^n$, the following are equivalent:
	\begin{enumerate}
		\item[(1)] $X$ is closed.
		\item[(2)] For every definable lower semi-continuous function $f:X \to [0,\infty)$, there exists $x_0 \in X$ such that $S_f(x_0)=\{x_0\}$, where $S_f(x) :=\{y \in X\;|\;\|x-y\| \leq f(x)-f(y)\}$ for every $x \in X$.
	\end{enumerate}
\end{theorem}

\section{Lipschitz curve selection}\label{sec:curve_selection}
Thamrongthanyalak used Fischer's $\Lambda^m$-regular stratification \cite{Fischer} in his paper \cite{Thamrongthanyalak}.
In locally o-minimal structures, such a stratification is not always available.
For our purpose, a weaker substitute called Lipschitz curve selection (Lemma \ref{lem:Lipschitz}) is enough.
 
 Let $X$ and $T$ be definable sets.
 The parameterized family $\{S_t\;|\;t \in T\}$ of definable subsets of $X$ is called \textit{definable} if the union $\bigcup_{t \in T} \{t\} \times S_t$ is definable.
 For a set $X$, a definable family $\mathcal C$ of subsets of $X$ is called a \textit{definable filtered collection} if, for any $B_1, B_2 \in \mathcal C$, there exists $B_3 \in \mathcal C$ with $B_3 \subseteq B_1 \cap B_2$. 
 We say that $X$ is \textit{definably compact} if, for every definable filtered collection $\mathcal C$ of nonempty closed subsets of $X$, $\bigcap_{C \in\mathcal C} C$ is non-empty.
 
 \begin{lemma}\label{lem:compact}
 	A definable subset $X$ of $F^n$ is bounded and closed if and only if it is definably compact.
 \end{lemma}
 \begin{proof}
 	See \cite[Proposition 3.10]{J}, which proves the lemma when the
 	structure is o-minimal; the same proof works when the structure is definably
 	complete.
 \end{proof}
 
We introduce the notations used in the proof of Lemma \ref{lem:Lipschitz}.
Let $\mathbb M_n$ be the set of $n \times n$ matrices with entries in $F$.
Set $\mathbb H_{n,d}:=\{A \in \mathbb M_n\;|\; {}^t\!A=A, A^2=A, \operatorname{tr}(A)=d\}$
and $\mathbb H_n=\bigcup_{d=0}^n \mathbb H_{n,d}$.
For every linear subspace $H$ of $F^n$ of dimension $d$, we can find $A \in \mathbb H_{n,d}$ such that $H=AF^n$ and $A$ is the linear projection of $F^n$ onto $H$.
The algebraic set $\mathbb H_{n,d}$ is bounded and closed in $F^{n^2}$.
By Lemma \ref{lem:compact}, $\mathbb H_{n,d}$ is definably compact.
Let $\|A\|$ be the Euclidean norm of a matrix $A \in \mathbb M_n$ under the natural identification of $\mathbb M_n$ with $F^{n^2}$.
We define $\|v\|$ for $v \in F^n$ similarly, and $\|A\|_{\text{op}}:=\sup_{\|v\|=1}\|Av\|$.
The function $\delta:\mathbb H_n \times \mathbb H_n \to F$ is given by $\delta(A,B)=\|B^\perp A\|_{\text{op}}$, where $B^{\perp}=I_n-B$ and $I_n$ is the identity matrix of size $n \times n$.

\begin{lemma}[Lipchitz definable curve selection]\label{lem:Lipschitz}
Suppose that $\mathcal F$ is locally o-minimal.
Let $X$ be a definable subset of $F^n$ and $a \in \partial X$, where $\partial X$ is the frontier of $X$.
Then there exists a definable injective Lipschitz continuous map $\gamma:[0,d] \to F^n$ such that $\gamma(0)=a$ and $\gamma((0,d]) \subseteq X$.
\end{lemma}
\begin{proof}
		We assume that $a$ is the origin of $F^n$ for simplicity.
By \cite[Lemma 5.16]{Fornasiero_locally}, there exist $d'>0$ and a definable continuous map $f:[0,d') \to X$ such that $f(0)=a$ and $f(t) \in X$ for $t >0$.
We may assume that $f$ is injective by \cite[Theorem 5.1]{Fornasiero_locally}.
The definable set $M:=f((0,d'))$ is decomposed into finitely many definable $\mathcal C^1$ submanifolds using \cite[Theorem 5.11]{Fornasiero_locally} in the same manner as \cite[Theorem 5.6]{Fornasiero_locally}.
For some $e>0$, $f((0,e))$ coincides with one of them by local o-minimality.
By setting $d'=e$, we may assume that $M$ is a definable $\mathcal C^1$ submanifold of $F^n$.
We have $\dim M=1$ by \cite[Proposition 2.8(6)]{FKK}.

Fix a sufficiently small $\varepsilon>0$.
Let $\tau:M \to \mathbb  H_{n,1}$ be the definable continuous map sending $x \in M$ to the matrix which represents the projection onto the tangent space of $M$ at $x$.
Since $\mathbb H_{n,1}$ is definably compact, \cite[Theorem 4.5]{Fuji-compact} yields $A:=\lim_{t \to 0}\tau(f(t))$.
By linear change of coordinates, we may assume that $AF^n=F \times \{0\}^{n-1}$.
Let $\overline{\tau}:M \cup \{a\} \to \mathbb H_{n,1}$ be the extension of $\tau$ given by $\overline{\tau}(a)=A$.
The map $\overline{\tau} \circ f$ is continuous.
Therefore, we may assume that $\|\tau(f(t))-A\|<\varepsilon$ for $0<t<d'$ by taking a smaller $d'$ if necessary. 
In addition, the tangent space $\tau(f(t))F^n$ of $M$ at $f(t)$ is not orthogonal to $F \times \{0\}^{n-1}$ because $\varepsilon$ is sufficiently small.

Let $\pi:F^n \to F$ be the coordinate projection onto the first coordinate.
For every subset $S$ of $F^n$ and $u \in F$, we denote $S \cap \pi^{-1}(u)$ by $S_u$.
We show that $M_u$ is closed and discrete for every $u \in \pi(M)$. 
For contradiction, assume that $\dim M_u>0$ for some $u \in \pi(M)$.
We have $\dim f^{-1}(M_u)>0$ by \cite[Proposition 2.8(6)]{FKK}.
A nonempty open interval $I$ is contained in $f^{-1}(M_u)$.
The tangent space of $M$ at $f(t)$, $t \in I$, is orthogonal to the first coordinate axis, which is absurd.
We have shown that $\dim M_u=0$ for every $u \in \pi(M)$.
This implies that $M_u$ is discrete and closed by \cite[Proposition 2.8(1)]{FKK}.
If $0 \in \pi(M)$, the set $f^{-1}(M \cap \pi^{-1}(0))$ is discrete and closed by \cite[Proposition 2.8(1),(6)]{FKK}.
Therefore, we may assume that $0 \notin \pi(M)$ by taking a smaller $d'>0$.

Let $N:=f([0,d'/2])$. 
Observe that $N$ is definably compact by \cite[Proposition 1.10]{M} and Lemma \ref{lem:compact}.
Observe that $a \in N$ and $0 \in \pi(N)$. 
We show $0 \in \mycl(\pi(N) \setminus \{0\})=\mycl(\pi(N \setminus \{a\}))$, where $\mycl(\cdot)$ denotes the closure in $F$.
Assume for contradiction that $0 \notin \mycl(\pi(N) \setminus \{0\})$.
By local o-minimality, $0$ is isolated in $\pi(N)$.
We have $\pi(N)=\{0\}$ because $\pi(N)$ is definably connected by \cite[Corollary 1.5]{M}.
This deduces that $\dim M \cap \pi^{-1}(0) \geq \dim N \cap \pi^{-1}(0)=\dim N=1$, which is a contradiction.
We have shown that $0 \in \mycl(\pi(N) \setminus \{0\})$.
This deduces, by local o-minimality, that a closed interval one of whose endpoints is $0$ is contained in $\pi(N)$.
After a linear transformation, we may assume that the closed interval $[0,d]$ is contained in $\pi(N)$ for some $d>0$.

By \cite[Lemma 5.15]{Fornasiero_locally}, we can find a definable map $\gamma:[0,d] \to N$ such that $\pi(\gamma(u))=u$ for $0 \leq u \leq d$.
We may assume that $\gamma|_{(0,d)}$ is of class $\mathcal C^1$ and $\gamma|_{(0,d]}$ is continuous by \cite[Theore 5.11]{Fornasiero_locally} by taking a smaller $d>0$.
$\gamma$ is continuous at $0$.
In fact, by \cite[Theorem 4.5]{Fuji-compact}, $\lim_{t \to 0}\gamma(t)$ exists in the definably compact set $N$.
The continuity of $\pi$ implies $\pi(\lim_{t \to 0}\gamma(t))=\lim_{t \to 0} \pi(\gamma(t))=0$.
This implies that $\lim_{t \to 0}\gamma(t)=a=\gamma(0)$ because $N \cap \pi_1^{-1}(0)$ is the singleton $\{a\}$.

We prove $\|J_{\gamma}(t)\| < {1}/{\sqrt{1-\varepsilon^2}},$
where $J_{\gamma}$ is the Jacobian matrix of $\gamma$.
Fix $0<t<d$.
Let $V=\tau(\gamma(t))$ and $\omega:VF^n \to F$ be the linear bijection defined by $w(v)=\pi(v)$.
Let $e=(1)$ be the unit vector in $F$.
We have $J_{\gamma}(t) = \omega^{-1}(e)$.
%
%
By the definition of $M$ and \cite[Proposition 3.1(c)]{Fischer}, we have $\delta(V,A) \leq \|V-A\|<\varepsilon$.
This inequality and the inequality $\delta(V,A)+\delta(V,A^{\perp}) \geq 1$ imply $\|AV\|_{\text{op}}=\delta(V,A^{\perp})>\sqrt{1-\varepsilon^2}$.
Let $w$ be a unit vector in $VF^n$.
By the definitions of $V$ and $A$, we have $\|AV\|_{\text{op}}=\|\omega(w)\|$.
Therefore, $\|J_{\gamma}(t)\|=\|\omega^{-1}(e)\|=1/\|AV\|_{\text{op}}<1/\sqrt{1-\varepsilon^2}$.

Put $L=\sqrt{n/(1-\varepsilon^2)}$.
We show that $\gamma$ is $L$-Lipchitz continuous.
To show this, fix $0 \leq b_1 < b_2 \leq d$.
Let $\gamma_i(t)$ be the $i$th coordinate of $\gamma(t)$.
The mean value theorem is deduced only from the extreme value theorem.
Therefore, the mean value theorem holds in $\mathcal F$ by \cite[p.1786]{M}.
Apply it to $\gamma_i$, then we have $|\gamma_i(b_2)-\gamma_i(b_1)| \leq \sup_{b_1<t<b_2}|\gamma_i'(t)| (b_2-b_1) \leq \|J_{\gamma}\|(b_2-b_1)<(b_2-b_1)/\sqrt{1-\varepsilon^2}$.
This deduces $\|\gamma(b_2)-\gamma(b_1)\| \leq L(b_2-b_1)$, which means that $\gamma$ is $L$-Lipchitz continuous.
\end{proof}

\begin{proof}[Proof of Theorem \ref{thm:local_case}]
	For every non-closed definable subset $E$ of $F^n$, we have only to construct a definable contraction on $E$ that has no fixed point.
	Let $a \in \partial E$.
	By Lemma \ref{lem:Lipschitz}, we can pick an $L$-Lipschitz definable map $\gamma:[0,d] \to E \cup \{a\}$ for some $L>0$ such that $\gamma(0)=a$ and $\gamma((0,d]) \subseteq E$.
	Let $H:E \to E$ be the definable map defined by $H(x)=\gamma(\min(d,\|x-a\|)/2L)$.
	We can show that $H$ is a definable contraction on $E$ having no fixed point in the same manner as the proof of \cite[Theorem A]{Thamrongthanyalak}.
	We omit the details.
\end{proof}

\section{Caristi fixed point theorem}\label{sec:caristi}

We prove Theorem \ref{thm:caristi} in this section.
\begin{lemma}\label{lem:lower_semi}
	Let $X$ be a definable, bounded and closed subset of $F^n$ and $f:X \to [0,\infty)$ be a definable lower semi-continuous function.
	Then, $\inf f(X) \in f(X)$. 
\end{lemma}
\begin{proof}
	Set $T:=f(X)$.
	For every $t \in T$, put $C_t:=\{x \in X\;|\; f(x) \leq t\}$.
	$C_t$ is a closed subset of $X$ because $f$ is lower semi-continuous.
	Consider the family $\mathcal C=\{C_t\;|\; t \in T\}$, which is a definable filtered collection.
	By Lemma \ref{lem:compact}, there exists $x_0 \in \bigcap_{t \in T}C_t$.
	We have $f(x_0) \leq t$ for every $t \in T$.
	This implies that $f(x_0)=\inf T$.
\end{proof}

\begin{proof}[Proof of Theorem \ref{thm:caristi}]
	We denote $S_f(x)$ by $S(x)$ in the proof.
	
	$(1) \Rightarrow (2):$
	We first show that $S(y) \subseteq S(x)$ whenever $y \in S(x)$.
	We pick an arbitrary $z \in S(y)$.
	We have $\|z-x\| \leq \|y-z\|+\|x-y\| \leq f(z)-f(y) + f(y)-f(x)=f(z)-f(x)$.
	This means that $z \in S(x)$.
	
	We next reduce to the case where $X$ is bounded in $F^n$.
	Take an arbitrary point $x' \in X$.
	Put $X'=S(x')$.
	The definable closed set $X'$ is bounded because, for every element $z \in X'$, we have $\|z-x'\| \leq f(x')$.
	For every $x \in X'$, we have $S_f(x) \subseteq S(x')=X'$.
	Therefore, we may assume that $X$ is bounded by replacing $X$ with $X'$.
	
	Since $X$ is closed and bounded, we can find $x_0 \in X$ such that $f(x_0)=\inf f(X)$ by Lemma \ref{lem:lower_semi}.
	We have $f(y) \geq f(x_0)$ for every $y \in X$.
	This implies that $S(x_0)=\{x_0\}$.
	
	$(2) \Rightarrow (1):$
	Assume that $X$ is not closed.
	Let $p$ be a point in the frontier of $X$.
	We define the definable function $f:X \to [0,\infty)$ by $f(x)=2\|x-p\|$.
	$f$ is continuous, and it is also lower semi-continuous.
	We show that $S(x) \neq \{x\}$ for every $x \in X$.
	We fix $x \in X$ to show this relation.
	Let $y$ be an arbitrary point in $X$.
	Triangle inequality implies $\frac{1}{2}(f(x)+f(y)) \geq \|x-y\|$.
	Therefore, we have $f(x)-f(y) \geq \|x-y\| + \frac{1}{2}(f(x)-3f(y))$.
	If we choose $y_0 \in X$ sufficiently close to $p$, we have $3f(y_0) <f(x)$.
	We obtain $y_0 \in S(x)$ and $S(x) \neq \{x\}$. 
\end{proof}


\begin{thebibliography}{99}
\bibitem{Caristi}
J. Caristi,
\emph{Fixed point theorems for mappings satisfying inwardness conditions},
Trans. Amer. Math. Soc., \textbf{215} (1976), 241-251.
	
\bibitem{DMS}
A. Dolich, C. Miller and C. Steinhorn,
\emph{Structure having o-minimal open core},
Trans. Amer. Math. Soc., \textbf{362} (2010), 1371-1411.
	
\bibitem{D}
L.~van den Dries, 
\emph{Tame topology and o-minimal structures},
London Mathematical Society Lecture Note Series, Vol. 248.
Cambridge University Press, Cambridge, 1998.	

\bibitem{Fischer}
A. Fischer,
\emph{O-minimal $\Lambda^m$-regular stratification},
Ann. Pure Appl. Logic, \textbf{147} (2007), 101--112.

\bibitem{Fornasiero_locally}
A. Fornasiero,
\emph{Locally o-minimal structures and structures with locally o-minimal open core},
Ann. Pure Appl. Logic, \textbf{164} (2013), 211--229.
	
%

\bibitem{Fuji-compact}
M. Fujita, 
\emph{Definable compactness in definably complete locally o-minimal structures},
Fund. Math., \textbf{267} (2024) 129--156.

\bibitem{FKK}
M. Fujita, T. Kawakami  and W. Komine,
\emph{Tameness of definably complete locally o-minimal structures and definable bounded multiplication},
Math. Logic Quart.,  \textbf{68} (2022), 496-515.

\bibitem{J}
W. Johnson, 
\emph{Interpretable sets in dense o-minimal structures}, 
J. Symbolic Logic, \textbf{83} (2018), 1477–1500.

\bibitem{M}
C. Miller,
\emph{Expansions of dense linear orders with the intermediate value property},
J. Symbolic Logic, \textbf{66} (2001), 1783-1790.

\bibitem{MS}
C. Miller and P. Speissegger, 
\emph{Expansions of the real line by open sets: o-minimality and open cores},
Fund. Math., \textbf{162} (1999), 193-208.

\bibitem{Thamrongthanyalak}
	A. Thamrongthanyalak,
	\emph{Expansions of real closed fields with the Banach fixed point property},
	Math. Logic Quart.,  \textbf{70} (2024), 197-204.
	
\bibitem{TV}
C. Toffalori and K. Vozoris, 
\emph{Notes on local o-minimality},
Math. Logic Quart., \textbf{55} (2009), 617-632.

\bibitem{Weston}
J. W. Weston,
\emph{A characterization of metric completeness},
Proc. Amer. Math. Soc., \textbf{64} (1977), 186-188.
\end{thebibliography}
\end{document}